\documentclass[12pt]{article}
\usepackage{makeidx}
\usepackage{latexsym}
\usepackage[margin=2.5cm]{geometry}
\usepackage[dvips]{graphicx}
\usepackage[english]{babel}
\usepackage{amssymb}
\usepackage{amsmath}
\usepackage{amsthm}

\newtheorem{definition}{Definition}[section]
\newtheorem{theorem}[definition]{Theorem}
\newtheorem{lemma}[definition]{Lemma}

\newtheorem{proposition}[definition]{Proposition}
\newtheorem{claim}[definition]{Claim}

\newtheorem{conjecture}[definition]{Conjecture}

\theoremstyle{definition}
\newtheorem{remark}[definition]{Remark}

\begin{document}

\date{}
\title
{Coloring graphs of various maximum degree from random lists}

\author{
{\sl Carl Johan Casselgren\thanks{{\it E-mail address:} 
carl.johan.casselgren@liu.se
\, Part of the work done
while the author was a postdoc at University of Southern Denmark
and at Mittag-Leffler Institute.
Research supported by postdoctoral grants from
SVeFUM and Mittag-Leffler Institute.}} \\ 
Department of Mathematics,
Link\"oping University, \\SE-581 83 Link\"oping, Sweden}

\maketitle

\bigskip 
\noindent
{\bf Abstract.}
	Let $G=G(n)$ be a graph on $n$ vertices
	with maximum degree $\Delta=\Delta(n)$.
	Assign to each vertex $v$ of $G$ a list $L(v)$ of colors
	by choosing each list independently and uniformly at random from all 
	$k$-subsets of a color set
	$\mathcal{C}$ of size $\sigma= \sigma(n)$.
	Such a list assignment is called a 
	\emph{random $(k,\mathcal{C})$-list assignment}.
	In this paper, we are interested in
	determining
	the asymptotic probability (as $n \to \infty$)
	of the existence of a proper coloring $\varphi$ of $G$,
	such that $\varphi(v) \in L(v)$ for every vertex $v$ of $G$,
	a so-called $L$-coloring.
	We give various lower bounds on $\sigma$, in
	terms of $n$, $k$ and $\Delta$, which ensures
	that with probability tending to 1 as $n \to \infty$
	there is an $L$-coloring of $G$.
	In particular, we show, for all fixed $k$ and growing $n$,
	that if $\sigma(n) = \omega(n^{1/k^2} \Delta^{1/k})$
	and
	$\Delta=O\left(n^{\frac{k-1}{k(k^3+ 2k^2 - k +1)}}\right)$,
	then the probability
	that $G$ has an $L$-coloring tends to 1 as $n \rightarrow \infty$.
	If $k\geq 2$ and $\Delta= \Omega(n^{1/2})$, then the same conclusion
	holds provided that
	$\sigma=\omega(\Delta)$.
	We also give related results
	for other bounds on $\Delta$, when $k$ is constant
	or a strictly increasing function of $n$.
\bigskip


\section{Introduction}
	Given a graph $G$, assign to each vertex $v$ of $G$ a set
	$L(v)$ of colors (positive integers).
	Such an assignment $L$ is called a \emph{list assignment} for $G$ and
	the sets $L(v)$ are referred
	to as \emph{lists} or \emph{color lists}.
	If all lists have equal size $k$, then $L$
	is called a \emph{$k$-list assignment}.
	We then want to find a proper
	vertex coloring $\varphi$ of $G$,
	such that $\varphi(v) \in L(v)$ for all
	$v \in V(G)$. If such a coloring $\varphi$ exists then
	$G$ is \emph{$L$-colorable} and $\varphi$
	is called an \emph{$L$-coloring}. Furthermore, $G$ is
	called \emph{$k$-choosable} if it is $L$-colorable
	for every $k$-list assignment $L$.
	
	This particular variant of vertex coloring is
	known as \emph{list coloring}
	and was introduced by Vizing \cite{Vizing} and independently
	by Erd\H os et al. \cite{ERT}.
	
	In this paper we continue the study of the problem
	of coloring graphs from random lists introduced
	by Krivelevich and Nachmias \cite{Krivelevich1,
	Krivelevich2}:
	Assign lists of 
	colors to the vertices of a graph $G = G(n)$ with $n$ vertices
	by choosing for
	each vertex $v$ its list $L(v)$ independently and uniformly at random
	from all $k$-subsets of a
	color set $\mathcal{C} = \{1,2,\dots,\sigma\}$.
	Such a list assignment is called a 
	\emph{random $(k,\mathcal{C})$-list assignment} for $G$.
	Intuitively it should hold that the larger $\sigma$ is,
	the more spread are the colors chosen for the lists
	and thus the more likely it is that we can find
	a proper coloring of $G$ with colors from the lists.
	The question that we address in this paper is
	how large $\sigma = \sigma(n)$
	should be
	in order to guarantee that with high 
	probability\footnote{An event $A_n$ occurs 
	\emph{with high probability} (abbreviated {\bf whp})
	if $\lim_{n \to \infty} \mathbb{P}[A_n] = 1$.}
	(as $n \to \infty$) there is a proper coloring
	of the vertices of $G$ with colors 
	from the random list assignment.
	
	This problem was first studied by Krivelevich and Nachmias
	\cite{Krivelevich1, Krivelevich2}
	for the case of powers of cycles and
	the case of complete bipartite graphs 
	with parts of equal size $n$.
	In the latter case they showed that for all fixed $k \geq 2$,
	the property of having a proper coloring from a random
	$(k,\mathcal{C})$-list assignment exhibits a sharp threshold,
	and that the location of that threshold is exactly
	$\sigma(n) = 2n$ for $k=2$.
	In \cite{Casselgren}, we generalized the second part of this
	result and showed
	that for a complete multipartite graph with $s$ parts  (fixed $s \geq 3$)
	of
	equal size $n$, the property of having a proper coloring from
	a random $(2,\mathcal{C})$-list assignment, has a sharp
	threshold at $\sigma(n) = 2(s-1)n$.
		
	Let $C^r_n$ be the $r$th power of a cycle
	with $n$ vertices.
	For powers of cycles, Krivelevich and Nachmias proved the 
	following theorem establishing a coarse threshold for the property
	of being colorable from a random list assignment.
	
\begin{theorem}
\label{th:powercycle}
	\cite{Krivelevich1}
	Assume $r, k$ are fixed integers satisfying
	$r \geq k$ and let $L$ be a random $(k,\mathcal{C})$-list assignment
	for $C_n^r$. If we denote by 
	$p_C(n)$ the probability that $C_n^r$ is $L$-colorable, then
\[ p_C(n) = \left\{ 
	\begin{array}{ll}
	o(1), & \sigma(n) = o(n^{1/k^2}), \\
	1 - o(1), & \sigma(n) = \omega(n^{1/k^2}).
	\end{array} \right. 
	  \] 		
\end{theorem}

	In \cite{Casselgren2, Casselgren3} we generalized Theorem 
	\ref{th:powercycle}:

\begin{theorem}
\label{th:lists}
	\cite{Casselgren2, Casselgren3}
	Let $G = G(n)$ be a graph on $n$ vertices with
	maximum degree bounded by some absolute constant,
	$k$ a fixed positive integer, and
	$L$ a random $(k,\mathcal{C})$-list assignment for $G$.
	If $\sigma(n) = \omega(n^{1/k^2})$, then
	{\bf whp} $G$ is $L$-colorable.
\end{theorem}
	
	Note that Theorem \ref{th:lists} is best possible
	for graphs with bounded maximum degree.
	Further results on the problem of coloring graphs from random lists
	appears in \cite{AndrenCasselgrenOhman, CasselgrenHaggkvist}.

	In this paper we consider random $(k,\mathcal{C})$-list assignments
	for graphs $G=G(n)$ on $n$ vertices 
	whose maximum degree 
	$\Delta = \Delta(n)$ is an increasing function of $n$.
	We would like to suggest the following:
	
	\begin{conjecture}
	\label{conj:listdegree}
		Let $G=G(n)$ be a graph on $n$ vertices
	with maximum degree
	at most $\Delta=\Delta(n)$, $k \geq 2$ a fixed positive integer, 
	and $L$ a random $(k,\mathcal{C})$-list assignment for $G$.
		\begin{itemize}
		
			\item[(i)] If $\Delta = o\left(n^{\frac{1}{k^2-k}}\right)$ 
			and $\sigma = \omega\left(n^{1/k^2} \Delta^{1/k}\right)$,
			then {\bf whp} $G$ is  $L$-colorable.
			
			\item[(ii)] If $\Delta = \Omega\left(n^{\frac{1}{k^2-k}}\right)$
			and $\sigma = \omega\left(\Delta\right)$, then {\bf whp}
			$G$ is $L$-colorable.

		\end{itemize}
	\end{conjecture}
	
	Note that when $\Delta$ is bounded then part (i) of the conjecture
	reduces to Theorem \ref{th:lists}.
	Moreover, for the case $k=1$, it
	is easily seen that
	for a graph $G=G(n)$ on $n$ vertices with strictly increasing
	maximum degree $\Delta = \Delta(n)$, 
	the coarse threshold for colorability from a
	random $(1,\mathcal{C})$-list assignment trivially
	occurs at $\sigma = n \Delta$.

	We prove that part (i) of Conjecture \ref{conj:listdegree} is true
	for a slightly more restrictive bound on $\Delta$:

\begin{theorem}
\label{th:main1}
	Let $G=G(n)$ be a graph on $n$ vertices
	with maximum degree
	at most $\Delta=\Delta(n)$, $k$ a fixed positive integer, and $L$ a random
	$(k,\mathcal{C})$-list assignment for $G$. If 
	$\sigma(n) = \omega\left(n^{1/k^2} \Delta^{1/k}\right)$ and
	$\Delta = O\left(n^{\frac{k-1}{k(k^3+ 2k^2 - k +1)}}\right)$, 
	then {\bf whp} $G$ is $L$-colorable.
\end{theorem}
 
		For the case $k=2$ we prove that both part (i) and (ii) of
	Conjecture \ref{conj:listdegree} 
	is true:
	
\begin{theorem}
	\label{th:main2}
	Let $G=G(n)$ be a graph on $n$ vertices with maximum
	degree at most $\Delta = \Delta(n)$, and $L$ 
	a random $(2,\mathcal{C})$-list assignment for $G$.
	If 
	
	\begin{itemize}
		
		\item[(i)] $\Delta = o\left(n^{1/2}\right)$ and $\sigma =
		\omega\left(n^{1/4} \Delta^{1/2}\right)$,
		or
		
		\item[(ii)]
		$\Delta= \Omega\left(n^{1/2}\right)$ and
		$\sigma=\omega\left(\Delta\right)$,
		
	\end{itemize}
	then {\bf whp} $G$ is $L$-colorable.
\end{theorem}

	For a complete graph $K_n$ on $n$ vertices,
	the property of being colorable
	from a random $(2,\mathcal{C})$-list assignment
	has a sharp threshold at $\sigma(n) = 2n$ \cite{Casselgren2},
	and for $k \geq 3$,  $K_n$ is {\bf whp} colorable from a random 
	$(k, \mathcal{C})$-list assignment if
	$\sigma \geq 1.223n$ \cite{Casselgren3}.
	Thus Theorem \ref{th:main2} yields a better bound on
	$\sigma(n)$ for graphs $G$ with maximum degree $o(n)$.

	In Section 2 we shall prove Theorem \ref{th:main1}
	and also give an example which shows
	that part (i) of Conjecture \ref{conj:listdegree} 
	(and therefore also Theorem \ref{th:main1}) is sharp
	in the ``coarse threshold sense'': for each integer $n$
	and each integer-valued function $\Delta = \Delta(n)$ 
	satisfying $\Delta =O(n^{\frac{1}{k^2-k}})$,
	there is a graph $H=H(n)$ with maximum degree $\Delta$
	such that if
	$\sigma(n) = o(n^{1/k^2} \Delta^{1/k})$
	and $L$
	is a random $(k,\mathcal{C})$-list assignment for $H$,
	then {\bf whp} $H$ is not $L$-colorable.
	In Section 2 we also prove some related results for
	graphs with other bounds on the maximum degree. 
	
	Section 3 contains the proof of Theorem \ref{th:main2}.
	Note also that part (ii) of the conjecture
	(and thus Theorem \ref{th:main2})
	is best possible in the same sense as Theorem \ref{th:main1},
	since a clique on $\Delta$ vertices requires at least
	$\Delta$ colors for a proper coloring.

	In Section 4 we prove some related results for graphs with
	fixed girth greater than $3$, and
	we also prove a result for graphs with girth
	$\Omega(\log  \log n)$:
	for each constant $C >0$, there are
	constants $k_0 = k_0(C)$ and $B = B(C)$, such that
	if $G$ is a graph with $n$ vertices, maximum degree at most $\Delta$,
	and girth
	$g \geq C \log \log n $,
	$k \geq k_0$ and
	$\sigma(n) > B \Delta$,
	then {\bf whp} $G$ is colorable from a random $(k,\mathcal{C})$-list
	assignment.
	(For instance when $C = 1$, then $k_0 =9$ and $B = 81 e^{7/2}$.)

	In Section 5
	we consider random list assignments where the uniform list size
	$k$ is a strictly increasing function of $n$.
	In particular, we present an analogue of Theorem \ref{th:main1}
	for lists of non-constant size, and
	we prove that for any constant $C > 0$, there is a constant $A=A(C)$,
	such that if
	$k \geq C \log n$ and $\sigma \geq A \Delta  \log n$,
	then $G$ is {\bf whp}
	colorable from a random 
	$(k, \mathcal{C})$-list assignment.


\section{Proof of Theorem \ref{th:main1}}
	In this section we prove Theorem \ref{th:main1}
	and some related results.
	Our basic method in this paper is rather similar to the proof
	of the main result of \cite{Casselgren3}, but we need to refine
	the method introduced in that paper, and use sharper estimates
	in many counting arguments.
	
	Let $H$ be a graph and $L$ a list assignment for $H$.
	If $H$ is not $L$-colorable, but removing any vertex from
	$H$ yields an $L$-colorable graph, then $H$ is \emph{$L$-vertex-critical}
	(or just \emph{$L$-critical}).
	Obviously, if $L$ is a list assignment for a graph $G$,
	and $G$ is not $L$-colorable, then $G$ contains a connected 
	induced $L$-critical subgraph.
	
	Suppose now that $H - w_1$ is $L$-colorable, where $w_1$ is some
	vertex of $H$. 
	Given an $L$-coloring $\varphi$ of $H - w_1$,
	a path $P=w_1 w_2 \dots w_t$ in $H$
	is called \emph{$(\varphi,L)$-alternating}
	if there are colors $c_2, c_3, \dots, c_t$ such that
	$\varphi(w_i) = c_i$
	and $c_i \in L(w_{i-1})$, $i=2,\dots,t$.
	We allow such a path to have length $0$ and thus only
	consist of $w_1$.
	The set of vertices
	which are adjacent to a vertex $x$ in a graph $G$ is denoted by $N_G(x)$.
	The following lemma was proved in \cite{Casselgren3}.

\begin{lemma}
\label{lem:charac}
	Let $F$ be a graph and $L$ a list assignment for $F$.
	If $F$ is $L$-critical, then
	for any
	vertex $v_1 \in V(F)$,
	$F - v_1$ has an $L$-coloring $\varphi$
	that satisfies the following conditions:

\begin{itemize}

	\item[(i)] all vertices in $F$ lie on $(\varphi,L)$-alternating
	paths with origin at $v_1$;
	
	\item[(ii)] for each color $c \in L(v_1)$,
	there is a vertex $w \in N_F(v_1)$, such that $\varphi(w) = c$.

	\item[(iii)] Define a rank function $R : V(F) \to \{0,1,\dots, |V(F)|-1\}$
	on the vertices of $F$ by setting
	$R(u) = j$ if
	a shortest $(\varphi,L)$-alternating
	path from $v_1$ to $u$ has length $j$.
	Then for every vertex $x$ of $F - v_1$ and
	every color $c \in L(x) \setminus \{ \varphi(x) \}$, there is either
	\newline
	(a) a vertex $y \in N_F(x)$ colored $\varphi(y) = c$ or
	\newline
	(b) a vertex $z \in N_F(x)$ such that $c \in L(z)$
	and $R(z) < R(x)$.
\end{itemize}

\end{lemma}

	For a rank function $R$ defined as in part (iii) of
	Lemma \ref{lem:charac}, we say that
	$R$ is the rank function on $V(F)$
	\emph{induced by $L$ and $\varphi$}.
	
	Let $F$ be a connected induced subgraph of a graph $G$,
	$v_1$ a fixed vertex of $F$ and 
	$R : V(F) \to \{0,1,\dots, |V(F)|-1\}$ a rank function
	on the vertices of $F$. 
	The triple $(F,v_1,R)$
	is \emph{proper}, if $R(v_1) = 0$
	and $R(u) > 0$ for each vertex $u \in V(F) \setminus \{v_1\}$,
	and if $R(u) = s$, then there is a vertex $x \in N_F(u)$ 
	such that $R(x) = s-1$.
	(This definition of proper triple is slightly different
	from the one used in \cite {Casselgren3}.)
	We also say that $(F,v_1,R)$ is a \emph{proper triple of $G$}.
	Note that if $F,v_1$ and $R$ satisfies the conditions of
	Lemma \ref{lem:charac} for some choice of $L$ and $\varphi$, 
	then $(F,v_1,R)$ is proper.
	The next lemma gives an upper bound on the number of proper triples
	in a graph.

\begin{lemma}
\label{lem:numbersubgraphs} 
	Let $G$ be a graph on $n$ vertices whose maximum degree is
	at most $\Delta$.
	The number of proper triples $(F, v_1, R)$,
	such that $F$ is a subgraph of $G$ with $m$ vertices
	does not
	exceed $$n \Delta^{m-1} (m-1)!.$$
\end{lemma}

\begin{proof}
	If $(F, v_1, R)$ is a proper triple, then by removing some edges
	of $F$ we can construct a tree $T$ with root $v_1$, such that if
	$u \in V(F)$ has rank $R(u) = r$ ($1\leq r \leq m-1$),
	then $u$ is adjacent to a vertex $x$ in $T$ with rank $R(x) = r-1$,
	and $v_1$ is the unique vertex of rank $R(v_1) = 0$.
	
	Moreover, given such a tree $T$ in $G$ with root $v_1$ and
	with a rank function
	$R : V(T) \to \{0,1,\dots, m-1\}$ satisfying these conditions,
	there is a uniquely determined
	proper triple $(F,v_1, R)$ with $V(F) = V(T)$, because $F$ is an
	induced subgraph of $G$. Hence, the number of proper triples of 
	$G$ is bounded by the number of such trees in $G$ together
	with a rank function $R$.
	This latter quantity is bounded by
	$$n \Delta (2 \Delta) (3 \Delta) \dots 
	((m-1)\Delta) = n \Delta^{m-1} (m-1)!;$$
	because there are $n$ ways of selecting $v_1$, 
	then we have $\Delta$ choices
	for  a neighbor of $v_1$ as the next vertex $v_2$ of $T$;
	thereafter, there are at most $(2 \Delta)$ ways of choosing
	an edge incident with $v_1$ or $v_2$ that connects one
	of these vertices with the next vertex of $T$, etc. 
\end{proof}
	
	Given a proper triple $(F, v_1, R)$ of a graph $G$ and
	a list assignment $L$ for $G$ such that $F$ is not $L$-colorable,
	we say that
	$(F,v_1,R)$ is \emph{$L$-bad} (or just {\em bad})
	if there is an $L$-coloring $\varphi$ of $F - v_1$,
	such that $F, v_1, R, L$ and $\varphi$
	satisfy conditions (i)-(iii) of Lemma \ref{lem:charac}.
	In particular,
	$R$ is the rank function on $V(F)$ induced by $L$ and $\varphi$.
	
	Consider  a random $(k,\mathcal{C})$-list assignment for
	a graph $G$,
	where $\mathcal{C} = \{1, 2, \dots,\sigma \}$.
	The following lemma was proved in \cite{Casselgren3}.

\begin{lemma}
\label{lem:bad}
	Let $L$ be a random $(k, \mathcal{C})$-list assignment for
	a graph $G$ with maximum degree at most $\Delta$.
	If $(F, v_1, R)$ is a proper triple of $G$ with $m = |V(F)|$, then
	$$\mathbb{P}[(F,v_1,R) \text{ \emph{is $L$-bad}}] \leq
	\frac
	{\sigma^{m-1} \binom{\Delta}{k} \binom{\Delta k}{k-1}^{m-1} }
	{\binom{\sigma}{k}^m}.$$
\end{lemma}

	We are now in position to prove Theorem \ref{th:main1}.
	
\begin{proof}[Proof of Theorem \ref{th:main1}.]
	Let $G=G(n)$ be a graph on $n$ vertices with maximum
	degree at most $\Delta$ and let $L$ be a random
	$(k,\mathcal{C})$-list assignment for $G$, where
	$k$ is a fixed positive integer,
	and suppose further that
	$\Delta = O(n^{\frac{k-1}{k(k^3+ 2k^2 - k +1)}})$
	and $\sigma(n) = \omega(n^{1/k^2}\Delta^{1/k})$.
	As pointed out above, the theorem is trivially true
	in the case when $k=1$. So in the following we shall
	assume that $k \geq 2$.
	We will show that {\bf whp} $G$ has no connected
	induced $L$-critical subgraph.
	This suffices for proving the theorem.
	By Lemma \ref{lem:charac}, this means that we need to to prove that
	if $\sigma(n) = \omega(n^{1/k^2}\Delta^{1/k})$, then
	{\bf whp} $G$ does not contain an $L$-bad proper triple
	$(F,v_1,R)$.
	We will use first moment calculations.
	
	If $(F,v_1, R)$ is $L$-bad, then
	$F$ has at least $k+1$ vertices.
	We first consider the case when $F$ has exactly $k+1$ vertices.
	It is not hard to see that if $F$ is not $L$-colorable
	and $|V(F)| = k+1$, then
	$F$ is a $(k+1)$-clique where all vertices
	have identical lists.
	The number of $(k+1)$-cliques in $G$ is at most $n \Delta^k$.
	Thus the expected number of $(k+1)$-cliques where the
	vertices get identical lists is at most
	
	\begin{equation}
	\label{eq:probcliques}
	n \Delta^k \binom{\sigma}{k}^{-k},
	\end{equation}
	which tends to $0$ as $n \to \infty$,
	because $\sigma(n) = \omega(n^{1/k^2} \Delta^{1/k})$.
	Hence, {\bf whp} there is no
	bad proper triple $(F,v_1, R)$ in $G$ satisfying that $|V(F)|=k+1$.
	
	Let us now consider the case when $|V(F)| \geq k+2$.
	First we show that if $(F,v_1, R)$ is bad,
	then {\bf whp} $F$ contains at most
	$\Delta^{k^2+k}$ vertices.
	Consider a path $P$ on $r$ vertices in $G$ with origin at some vertex $v$.
	The probability that
	there is an $L$-coloring $\varphi$ of $P-v$, such that $P$ is
	$(\varphi,L)$-alternating is at most
	
$$
 \frac{\sigma(\sigma-1)^{r-2}\binom{\sigma-1}{k-1}^2 	
		\binom{\sigma-2}{k-2}^{r-2}}
	{\binom{\sigma}{k}^{r}} \leq \frac{k^{2r}}{\sigma^{r-1}},
$$
	because there are at most $\sigma(\sigma-1)^{r-2}$ ways of choosing the
	proper coloring $\varphi$ and thereafter at most
	$\binom{\sigma-1}{k-1}^2 \binom{\sigma-2}{k-2}^{r-2}$ ways of choosing
	the list assignment $L$ for $P$ so that $\varphi$ is an $L$-coloring
	of $P-v$ and $P$ is $(\varphi,L)$-alternating.
	Moreover, the number of distinct paths in $G$ on $r$ vertices
	is at most $n \Delta^{r-1}$.
	Therefore,
	the expected number of paths $P$ in $G$
	on at least
	$k^2+k+1$ vertices, for which there is an $L$-coloring $\varphi$ of $P-v$
	such that $P$ is $(\varphi, L)$-alternating is at most
	
	\begin{equation}
\label{eq:pathsum}
	 \sum_{r=k^2+k+1}^{n} \frac{n \Delta^{r-1}k^{2r}}{\sigma^{r-1}}
	= o(1),
	\end{equation}
	since $\Delta = O(n^{\frac{k-1}{k(k^3+ 2k^2 - k +1)}})$ and
	$\sigma(n) = \omega(n^{1/k^2} \Delta^{1/k})$. Hence, by Markov's inequality,
	{\bf whp} there is no
	$L$-coloring $\varphi$ of a subgraph of $G$
	such that $G$ contains a
	$(\varphi, L)$-alternating path of length	$k^2+k+1$.
	
	Now, by Lemma \ref{lem:charac}, if $F$ is a subgraph of $G$ that belongs
	to a bad proper triple $(F,v_1, R)$,
	then there is an $L$-coloring $\varphi$ of $F-v_1$
	such that all vertices of
	$F$ lie on $(\varphi,L)$-alternating paths with origin at $v_1$.
	Since {\bf whp}
	the maximum length of such a path in $G$ is at most $k^2+k$, the
	maximum number of vertices in a subgraph of $G$ that 
	is in a bad proper triple is {\bf whp} at most
	\[ 1 + \Delta + \Delta^2 + \dots + \Delta^{k^2+k-1} \leq \Delta^{k^2+k}. \]

	Let $X_m$ be a random variable counting the number
	of bad proper triples $(F,v_1, R)$ in $G$ such that $F$ has $m$ vertices
	and set $$X = \sum_{m=k+2}^{\Delta^{k^2+k}} X_m.$$	
	Lemma \ref{lem:bad} gives an upper bound on the probability
	that a given proper triple of $G$ on $m$ vertices 
	is bad.
	Additionally, by Lemma \ref{lem:numbersubgraphs},
	$$f(m) \leq n \Delta^{m-1}(m-1)!,$$ where
	$f(m)$ is the number of proper triples $(F, v_1, R)$ in $G$
	such that	$F$ has $m$ vertices.
	Let $p_m$ be the least number
	such that $$\mathbb{P}[(F, v_1, R) \text{ is $L$-bad}] \leq p_m,$$
	whenever $(F, v_1, R)$ is a proper triple in $G$
	and $F$ has $m$ vertices.
	Since such a subgraph $F$ {\bf whp}
	has at most $\Delta^{k^2+k}$
	vertices if $(F, v_1, R)$ is $L$-bad, we conclude from
	Lemmas
 	\ref{lem:numbersubgraphs} and \ref{lem:bad} that
	
\begin{align}
	\mathbb{P[\text{$G$ contains an $L$-bad proper triple}]} & \leq
	\mathbb{E}[X] + o(1) \nonumber \\
	&
	\leq \sum_{m=k+2}^{\Delta^{k^2+k}} f(m) p_m + o(1)
	\nonumber \\
	&
	\leq \sum_{m=k+2}^{\Delta^{k^2+k}} 
	n \Delta^{2m-2}(m-1)!
	\frac
	{\sigma^{m-1} \binom{\Delta}{k} \binom{\Delta k}{k-1}^{m-1} }
	{\binom{\sigma}{k}^m} + o(1)
	\nonumber \\
	&	=O\left(\frac{n}{\sigma \Delta}\right) \sum_{m=k+2}^{\Delta^{k^2+k}} 
	 \Big(\frac{m k^k \Delta^{k}}
	{\sigma^{k-1}} \Big)^m
	\nonumber \\
	&
	= O \Big( \frac{n \Delta^{k(k+2)}}{\sigma^{(k-1)(k+2)+1}} \Big)
	\sum_{m=0}^{\infty} \Big(\frac{k^{k} \Delta^{k^2 + 2k}}
	{\sigma^{k-1}} \Big)^m
	\nonumber \\
	& = o(1), \nonumber 
\end{align}
	provided that $\sigma(n) = \omega(n^{1/k^2}\Delta^{1/k})$,
	$\Delta = O(n^{\frac{k-1}{k(k^3+ 2k^2 -k+1)}})$ and $k \geq 2$.
\end{proof}

	We now show that the bound on $\sigma$ in 
	part (i) of Conjecture \ref{conj:listdegree}
	(and also Theorem \ref{th:main1}) is best possible
	in the ``coarse threshold sense''.
	We will show that for positive integers $k \geq 2$ and 
	$n \geq k+1$ (large enough),
	and each non-constant increasing integer-valued
	function $\Delta= O(n^{\frac{1}{k^2-k}})$,
	there is a graph $H = H(n, \Delta)$ with
	$n$ vertices and maximum degree $\Delta$
	such that if $\sigma(n) = o(n^{1/k^2}\Delta^{1/k})$
	and $L$ is a random $(k,\mathcal{C})$-list assignment
	for $H$, then {\bf whp} $H$ is not $L$-colorable.
	
	So fix $k$, and let $n$ be a positive integer satisfying
	$n \geq k+1$ and assume that $\Delta$ satisfies
	$\Delta= O(n^{\frac{1}{k^2-k}})$.
	We set
	$n_{\Delta} = \lfloor \frac{n}{\Delta+1} \rfloor$,
	and let $H$ be a graph on $n$ vertices which is 
	the disjoint union of $n_{\Delta}$
	complete graphs, each of which has $\Delta+1$ vertices,
	and possibly some isolated vertices.
	Let $J_1, \dots, J_{n_{\Delta}}$ be the non-trivial components of $H$
	and let $L$ be a random $(k,\mathcal{C})$-list assignment
	for $H$.
	We will prove that 
	{\bf whp} there is at least one $(k+1)$-clique in $H$
	where all vertices have identical lists, which means
	that {\bf whp} $H$ is not $L$-colorable.
	Note that we may assume that $\Delta \leq \sigma$, since
	otherwise, trivially there is no $L$-coloring of $H$.
	
	Let $X$
	be a random variable counting the number of $(k+1)$-cliques
	in $H$ where all vertices have identical lists. Then

\begin{equation}	
\label{eq:expect}
\mathbb{E}[X] = \left \lfloor \frac{n}{\Delta+1} \right\rfloor
	\binom{\Delta+1}{k+1} \binom{\sigma}{k}^{-k} 
	=\Theta(n (\Delta+1)^k \sigma^{-k^2}).
\end{equation}

	To prove that $\mathbb{P}[X > 0] = 1 - o(1)$ we use the 
	second moment method with the 
	inequality due to Chebyshev in the following form:

\begin{equation}
\label{eq:chebyshev}
	\mathbb{P}[Y=0]	\leq \frac{\text{Var}[Y]}{\mathbb{E}[Y]^2},
\end{equation}
	valid for all non-negative random variables $Y$. 
	Since $X$ is a sum of indicator random variables, 
	we can use the following approach from 
	\cite{AlonSpencer}.
	
	Let $X = X_1 + \dots + X_d$, where each $X_i$ is the indicator
	random variable for the event that
	the vertices of a $(k+1)$-clique gets identical lists.
	Let $A_i$ be the event corresponding to
	$X_i$, that is, $X_i=1$ if $A_i$ occurs and $X_i=0$ otherwise.
	For indices $i,j$ we write $i \sim j$ 
	if $i \neq j$ and the events $A_i, A_j$ are not independent.
	Set
	
\begin{equation*}
	\Pi = \sum_{i \sim j} \mathbb{P}[A_i \wedge A_j].
\end{equation*}	
	When $i \sim j$, we have

\begin{equation*}
	\text{Cov}[X_i, X_j]= 
	\mathbb{E}[X_iX_j] - \mathbb{E}[X_i] \mathbb{E}[X_j] \leq
	\mathbb{E}[X_iX_j]= \mathbb{P}[A_i \wedge A_j]
\end{equation*}	
\medskip
	and when $i \neq j$ and not $i \sim j$ then Cov$[X_i,X_j]=0$. Thus
\begin{equation*}
	\text{Var}[X]\leq \mathbb{E}[X] + \Pi
\end{equation*}
	and the following proposition follows from \eqref{eq:chebyshev}.

\begin{claim}
\label{cl:cheb}
	
	If $\mathbb{E}[X] \rightarrow \infty$ 
	and $\Pi = o(\mathbb{E}[X]^2)$, then $\mathbb{P}[X > 0] = 1 - o(1)$.
	
\end{claim}
	It is clear from \eqref{eq:expect} that 
	$\mathbb{E}[X] \rightarrow \infty$
	if $\sigma(n) = o(n^{1/k^2}\Delta^{1/k})$. 
	We now show that the second
	criterion of Claim \ref{cl:cheb} is satisfied.
	If $i \sim j$, then clearly $A_i$ and $A_j$
	are events for cliques which are in the same component of
	$H$, and all vertices in these cliques have identical lists. 
	Since $H$ has $n_\Delta$ components
	and two distinct $(k+1)$-cliques have at most $k$ vertices
	in common we have
	
	$$ \Pi = O\left(n_\Delta\right) \sum_{l=1}^{k}
	\binom{\Delta+1}{k+1+l} \binom{\sigma}{k}^{-(k+l)}
	= O\left( n(\Delta+1)^{k+2} \sigma^{-k(k+1)}\right),$$
	and thus $\Pi = o(\mathbb{E}[X]^2)$ as required.
	We conclude that Theorem \ref{th:main1} is best possible
	in the ``coarse threshold sense''.

\bigskip
\bigskip

	Next, we will prove the following two propositions
	which show that weaker versions of
	Conjecture \ref{conj:listdegree} hold
	for larger $\Delta$.
	As usual, $G=G(n)$ is a graph on $n$ vertices with maximum
	degree at most $\Delta = \Delta(n)$, and 
	$L$ is a random $(k,\mathcal{C})$-list assignment for $G$,
	where $k \geq 2$ is a fixed positive integer.

\begin{proposition}
\label{prop1}
	Suppose that
	$\alpha$ and
	$s$ are constants satisfying $1 \leq \alpha \leq 3$ and
	and $s \geq 2 + \frac{2}{k-1}$,
	$\Delta = O(n^{1/k^{\alpha}})$, and
	$\sigma(n) = \omega\left(n^{1/k^{\frac{\alpha+1}{2}}} 
	\Delta^s \right)$.
	Then {\bf whp} $G$ is $L$-colorable.
\end{proposition}

\begin{proposition}
\label{prop2}
	If $\Delta = \Omega(n^{1/k})$ and
	$\sigma(n) = \omega(n^{\frac{1}{k}} \Delta)$,
	then {\bf whp} $G$ is $L$-colorable.
\end{proposition}

	\bigskip

	We first prove Proposition \ref{prop1}.
	The proof 
	is similar to the proof of Theorem \ref{th:main1}
	and therefore the proof will not be given in full detail.

\begin{proof}[Proof of Proposition \ref{prop1} (sketch).]
		
	Let $G = G(n)$ be a graph on $n$ vertices with
	maximum degree at most $\Delta = O(n^{1/k^\alpha})$, where
	$\alpha$ is a constant satisfying $1 \leq \alpha \leq 3$,
	and $k \geq 2$ is a fixed positive integer.
	Assume further that $s$ is a constant satisfying 
	$s \geq 2+ \frac{2}{k-1}$, and 
	$\sigma = \omega\left(n^{1/k^{\frac{\alpha+1}{2}}} \Delta^s \right)$.
	Note that the condition on $s$ implies that 
	\begin{equation}
	\label{eq:s}
		sk \geq 2k+s.
	\end{equation}
	
	As in the proof of Theorem \ref{th:main1},
	we shall prove that $G$ {\bf whp} has no bad
	proper triple.
	
	Proceeding as in the proof of Theorem \ref{th:main1},
	one may first deduce that
	any $(\varphi,L)$-alternating path
	in $G$ {\bf whp} has at most $k^{(\alpha+1)/2}$ 
	vertices,
	and thus if $(F,v,R)$ is a bad proper triple in $G$
	then {\bf whp } $F$ has at most 
	$\Delta^{k^{(\alpha+1)/2}}$ vertices.
	Using Lemmas \ref{lem:numbersubgraphs}
	and \ref{lem:bad} it now follows, as in the proof of
	Theorem \ref{th:main1}, that
	the probability that $G$ has a bad proper triple tends to zero
	if the sum
	$$
	O\left(\frac{n}{\sigma \Delta}\right) 
	\sum_{m=k+1}^{\Delta^{k^{\frac{\alpha+1}{2}}}} 
	 \Big(\frac{m k^k \Delta^{k}}
	{\sigma^{k-1}} \Big)^m
	$$
	tends to $0$ as $n \to \infty$.
	Rewriting this sum yields that it is at most
	\begin{equation}
	\label{eq:summand}
		O\left(
		\frac{n\Delta^{k(k+1)}} {\sigma^{k^2} \Delta}
		\right)
		\sum_{m=0}^{\infty} \left(
		\frac{\Delta^{k^{\frac{\alpha+1}{2}}} k^k \Delta^{k} \sigma}{\sigma^{k}}
		\right)^m.
	\end{equation}
	Setting $\sigma = n^{1/k^{\frac{\alpha+1}{2} }} \Delta^s$,
	we have that the ratio of the geometric sum in \eqref{eq:summand} is at most
	$$
	\frac{\Delta^{k} \Delta^s}{\Delta^{(s-1)k}}
	\frac{
	 \Delta^{k^{\frac{\alpha+1}{2}}} k^k n^{1/k^{\frac{\alpha+1}{2} }}} 
	{n^{1/k^{\frac{\alpha-1}{2} }} \Delta^k}.
	$$
	The first factor in this expression is bounded by \eqref{eq:s}. As regards
	the second factor, this quantity is maximum when
	$\Delta = \Theta(n^{1/k^\alpha})$, and using the 
	fact that  $\alpha -1 \leq \alpha/2 + 1/2$,
	when $\alpha \leq 3$, it
	follows that this factor is bounded as well.
	
	It follows that if 
	$\sigma = \omega\left(n^{1/k^{\frac{\alpha+1}{2}}} \Delta^s\right)$,
	then the ratio of the geometric sum in \eqref{eq:summand}
	tends to zero as $n \to \infty$, 
	and using this fact it is straightforward to verify that
	the expression \eqref{eq:summand} tends to zero.

\end{proof}

Let us now prove Proposition \ref{prop2}.

\begin{proof}[Proof of Proposition \ref{prop2} (sketch)]
	Let $k \geq 2$ be a positive integer and
	$G=G(n)$ be a graph on $n$ vertices with maximum degree
	at most $\Delta = \Omega(n^{1/k})$.
	We have to prove that
	if 
	$\sigma(n) = \omega(n^{1/k} \Delta)$ and
	$L$ is a random $(k,\mathcal{C})$-list assignment for $G$,
	then {\bf whp} $G$ is $L$-colorable.
	
	We will show that {\bf whp} $G$ contains no connected induced
	$L$-critical subgraph. 
	By Lemma \ref{lem:charac} it suffices to prove
	that {\bf whp} $G$ contains no vertex $v$ such that
	there are $k$ neighbors $u_1,\dots, u_k$ of $v$
	such that $c_i \in L(u_i)$, for $i=1,\dots,k$,
	where $L(v) =\{c_1, \dots, c_k\}$.
	The expected number of such vertices $v$ in $G$ is at most
	
	$$
	n \binom{\Delta}{k}\frac{\binom{\sigma}{k} k! \binom{\sigma-1}{k-1}^k}
	{\binom{\sigma}{k}^{k+1}} = O\left(\frac{n \Delta^k}{\sigma^k}   \right),
	$$
	which tends to $0$ as $n \to \infty$, so the desired result
	follows from Markov's inequality.
	\end{proof}


\section{Lists of size $2$}
	
	In this section we prove Theorem \ref{th:main2}.
	The proof is
	not very different from the proofs
	in the preceding section, but we use a somewhat
	different technique
	employed in \cite{Casselgren2} for proving
	results on random $(2,\mathcal{C})$-list assignments. Below we
	introduce some terminology and auxiliary results
	from that paper.
	
	Let $P = v_1 e_1 v_2 \dots e_{d-1} v_d$ be a path.
	Then the sequence $C = v_1 e_1 v_2 \dots v_d e_d v_1$ is an
	{\em ordered cycle} if $e_d = v_d v_1$.
	Similarly, the sequence $D = v_1 e_1 v_2 \dots v_d e_d v_j$
	is called an {\em ordered lollipop} if $e_d = v_d v_j$
	and $j \in \{2, \dots , d-2\}$.
	Note that an ordered cycle $C$ and an ordered lollipop $D$ 
	is uniquely determined by a sequence of 
	vertices (as is also a path).
	We may thus write $C = v_1 \dots v_d v_1$ for ordered cycles $C$, and
	similarly $D = u_1 \dots u_j \dots u_d u_j$, 
	for ordered lollipops $D$. 
	Paths, ordered cycles and ordered lollipops will usually be referred 
	to as sequences of vertices; however, sometimes
	we will refer to such sequences as graphs and then 
	mean the graph consisting of the vertices and edges of the sequence. 
	In particular, if $C$ 
	is an ordered cycle or lollipop, then
	$V(C)$ and $E(C)$ are the sets 
	of all vertices and edges in $C$, respectively.

	Let $L$ be a $2$-list assignment 
	for a graph $G$ and let $C = v_1 \dots v_d v_1$ be an ordered cycle of $G$. 
	Suppose that there are colors $c_1, \dots, c_{d-1}$ 
	such that $c_1 \in L(v_1)$, 
	$L(v_i) = \{c_{i-1}, c_i\}$, $i = 2, \dots , d - 1$ and
	$L(v_d) = \{c_{d-1}, c_1\}$. 
	Then $C$ is {\em $L$-alternating}. Similarly, an ordered lollipop 
	$D = u_1 \dots u_j \dots u_d u_j$ in $G$ is {\em $L$-alternating}
	if there are colors $c_1, \dots , c_{d-1}$ such that $c_1 \in L(u_1)$ 
	and $L(u_i) = \{c_{i-1}, c_i\}$, $i = 2, \dots , d-1$
	and $L(u_d) = \{c_{d-1}, c_j\}$.
	For an $L$-alternating ordered cycle or ordered lollipop $D$, 
	the common color $c_1$ of the lists of the
	first two vertices of $D$ will be referred to as the 
	{\em first color of $D$.}

	The following lemma was proved in \cite{Casselgren2}.
	
\begin{lemma}
\label{lem:2list}	
	Let $G$ be a graph and $L$ 
	a $2$-list assignment for $G$. If $G$ is not $L$-colorable, 
	then there are subgraphs $H_1$ and $H_2$ of G, such that for $i = 1, 2$:

\begin{itemize}
	
	\item[(i)] $H_i$ is either an $L$-alternating ordered cycle or an 
	$L$-alternating ordered lollipop;
	
	\item[(ii)] there is a vertex $v$ of $G$ with $L(v) = \{c_1, c_2\}$, 
	such that the first vertex of $H_i$ is $v$ and the 
	first color is $c_i$;
	
	\item[(iii)] the second vertex of $H_1$ and the second vertex of $H_2$ are distinct.

\end{itemize}
\end{lemma}

	Consider a graph $G$ with a $2$-list assignment $L$.
	The pair $F = (H_1,H_2)$ is then called a
	\emph{proper pair}
	if $H_1$ is an ordered cycle or ordered lollipop in $G$,
	$H_2$ is an ordered cycle or ordered lollipop in $G$ 
	and $H_1$ and $H_2$ have a common first vertex.
	Moreover, a proper pair $F = (H_1,H_2)$
	is \emph{$(L,2)$-bad} (or just {\em $2$-bad})
	if $H_1$,
	$H_2$ and $L$ satisfy
	the conditions (i)-(iii) of Lemma \ref{lem:2list}.
	Note that if $F=(H_1, H_2)$ is a $2$-bad proper pair, then $F$
	is not $L$-colorable.
	
	By slight abuse of terminology we will sometimes refer to
	proper pairs $F=(H_1, H_2)$ in $G$ as subgraphs of $G$.
	For such a proper pair $F = (H_1, H_2)$,
	$|V(H_1) \cup V(H_2)|$ is the number of vertices of $F$.

	\begin{remark}
		If $F = (H_1, H_2)$ is a $2$-bad proper pair
		with a common first vertex $v$, then trivially there
		is a subgraph $J$ of $F$ such that $J-v$ has an $L$-coloring $\varphi$.
		Moreover, if $R$ is the rank function induced by $L$ and $\varphi$
		as in Section 2, then $(J,v,R)$ is a bad proper triple. 
		So for a graph with
		a $2$-list assignment $L$, Lemma \ref{lem:2list}
		provides a stronger characterization
		than Lemma \ref{lem:charac} of which $2$-list assignments
		do not contain a proper coloring of the graph.
	\end{remark}

	Suppose that $H_1$ is an ordered cycle or lollipop
	and $H_2$ is an ordered cycle
	or lollipop. 
	Assume further that $H_1$ contains
	the vertices $v, v_2,\dots, v_{d_1}$
	and that the vertices lie in this order along $H_1$.
	Suppose that $v_i \in V(H_1) \cap V(H_2)$, $v_i \neq v$ 
	and let $u$ be the vertex that
	precedes $v_i$ along $H_2$. If $u v_i \notin E(H_1)$, then $v_i$
	is called a {\em non-consecutive common vertex} of $H_1$
	and $H_2$. Otherwise, if $uv_i \in H_1$ then
	$v_i$ is called a {\em consecutive common vertex}.
	
	The following two lemmas are essentially variants of
	Lemmas 13 and 14 in \cite{Casselgren2}, respectively;
	since we have not found
	a way to deduce them directly
	from those lemmas, we provide brief 
	sketches of the 
	proofs
	(for details see \cite{Casselgren2}).

\begin{lemma}
\label{lem:2bad}
	Let $G$ be a graph and let $L$
	be a random $(2,\mathcal{C})$-list assignment for $G$.
	Suppose that $F = (H_1, H_2)$
	is a proper pair in $G$ on $l$ vertices
	and $r$ non-consecutive common vertices.
	Then
	$$
	\mathbb{P}[F \text{ \emph{is $2$-bad}}] \leq
		\frac{2^{l+2r}}{\sigma^{l-1}(\sigma - 1)^{2+r}}.
	  $$
\end{lemma}

\begin{proof}[Proof (sketch).]
	Let $l_1 = |V(H_1)|$ and $l_2 = |V(H_2) \setminus V(H_1)|$.
	We will prove the lemma assuming that $l_2 > 0$. If $l_2 =0$,
	then a similar argument applies.
	Suppose that $H_1$ is an ordered cycle or an
	ordered lollipop on the vertices $v, v_2, v_3, \dots, v_{l_1}$
	and that the vertices lie in that order along $H_1$.
	Assume further that there are $r_1$ non-consecutive
	common vertices $u$ of $H_1$ and $H_2$
	such that the predecessor of $u$ along
	$H_2$ is in $V(H_1)$. Denote the set of these vertices by $R_1$.

	There are $\sigma (\sigma-1)$ ways of choosing the list
	$L(v)$ and selecting one of the colors of $L(v)$ as the first color
	of $H_1$.
	Then there are at most
	$(\sigma-1)^{l_1 - r_1- 2} 2^{2r_1}$
	ways of choosing the lists
	for $H_1 - v$, so that $H_1$ is $L$-alternating, and so that
	the list of every vertex of $R_1$
	has a color in common with the list of
	its predecessor along $H_2$; we choose the lists for
	the vertices $v_2, v_3, \dots, v_{l_1}$ sequentially
	except that for any pair of vertices $v_i$ and $v_j$ ($i < j$)
	that are adjacent on $H_2$ and satisfying that
	one of $v_i$ and $v_j$ is a non-consecutive common vertex, 
	we fix one color of $L(v_j)$
	immediately after choosing the colors for $L(v_i)$.
	
	So in total we have at most
	$\sigma (\sigma-1) (\sigma-1)^{l_1 - r_1- 2} 2^{2r_1}$
	choices for the lists of $H_1$.

	Since there are $r-r_1$ non-consecutive common vertices of $H_1$ and $H_2$
	that are not in $R_1$,
	there are thereafter at most $(\sigma-1)^{l_2 - 1 - (r-r_1)}2^{r-r_1}$
	ways of choosing the lists of the vertices
	in $V(H_2) \setminus V(H_1)$ so that $H_2$ is $L$-alternating as well.
	Since, in total, there are $\binom{\sigma}{2}^l$ ways of choosing 
	the restriction of $L$
	to the vertices in $F$, the result follows.
\end{proof}

\begin{lemma}
\label{lem:2numbersubgraphs}
	Let $G$ be a graph on $n$ vertices with maximum degree
	$\Delta$.
	The number of proper pairs $F = (H_1,H_2)$
	on $l$ vertices  in $G$ with $r$ non-consecutive vertices
	is at most $n \Delta^{l-1+r} 2^l$.
\end{lemma}

\begin{proof}[Proof (sketch).]
	Let $|V(H_1)| = l_1$ and $|V(H_2)\setminus V(H_1)| = l_2$.
	The first vertex $v$ of $H_1$ and $H_2$ can be chosen in at most
	$n$ ways. After that, there are at most $\Delta^{l_1-1}$
	choices for the rest of $H_1$.
	Thereafter, there are at most
	$2^{l_1} \Delta^{l_2+r}$
	choices  
	for the
	vertices of $H_2-v$,
	because there are at most $l_1$ vertices of $H_1$ which might be
	consecutive common vertices of $H_1$ and $H_2$,
	and $l_2+ r$ other vertices of 
	$H_2$. 
\end{proof}

\begin{proof}[Proof of Theorem \ref{th:main2}.]
	Let $G =G(n)$ be a graph on $n$ vertices with maximum
	degree at most $\Delta$ and let
	$L$ be a random $(2,\mathcal{C})$-list assignment
	for $G$. 
	By Lemma \ref{lem:2list}, it suffices to prove that
	if either $\Delta = \Omega(n^{1/2})$ and $\sigma = \omega(\Delta)$
	or $\Delta = o(n^{1/2})$ and $\sigma = \omega(n^{1/4} \Delta^{1/2})$,
	then {\bf whp} $G$ does not contain a $2$-bad proper pair.
	We will use easy first moment calculations.

	Any $2$-bad proper pair contains at least $3$ vertices.
	For $l \in \{3,\dots,n\}$ and $r \in \{1,\dots,l\}$,
	let $X_{l,r}$ be a random variable counting the number
	of $2$-bad proper pairs with $l$ vertices and
	$r$ non-consecutive common vertices in $G$,
	and let $$X = \sum_{l=3}^n \sum_{r=0}^l X_{l,r}.$$

	By Lemma \ref{lem:2numbersubgraphs},
	$f(l,r) \leq n \Delta^{l-1+r} 2^l$,
	where $f(l,r)$ is the number of
	proper pairs on $l$ vertices
	and $r$ non-consecutive common vertices in $G$.
	Let $p_{l,r}$ be the least number
	such that $\mathbb{P}[F_{l,r} \text{ is $2$-bad}] \leq p_{l,r}$, whenever
	$F_{l,r}$ is a proper pair on altogether $l$ vertices 
	and $r$ non-consecutive common vertices in $G$.
	By Lemmas \ref{lem:2bad} and \ref{lem:2numbersubgraphs} we have
\begin{align*}
	\mathbb{P[\text{$G$ contains a $2$-bad proper pair}]}  \leq
	\mathbb{E}[X] 
	& \leq \sum_{l=3}^{n} \sum_{r=0}^{l}  f(l,r) p_{l,r}
	\nonumber \\
	&
	\leq \sum_{l=3}^{n} \sum_{r=0}^{l}
	n \Delta^{l-1+r} 2^l
	\frac{2^{l+2r}}{\sigma^{l-1}(\sigma - 1)^{2+r}}
	\nonumber  \\
	&	=O\left(\frac{n \Delta^2}{\sigma^{4}}\right) \sum_{l=0}^{n} 
	 \left(\frac{4 \Delta}
	{\sigma} \right)^l
	\nonumber \\
	& = o(1),
\end{align*}
	provided that $\Delta = o(n^{1/2})$ and 
	$\sigma(n) = \omega(n^{1/4} \Delta^{1/2})$, or
	$\Delta = \Omega(n^{1/2})$ and $\sigma = \omega(\Delta)$.
\end{proof}

Consider a random $(2,\mathcal{C})$-list assignment
for a graph $G = G(n)$ on $n$ vertices with maximum degree
at most $\Delta = \Delta(n)$ and girth $g$.
A $2$-bad proper pair in $G$ has at least $g$ vertices, and by proceeding precisely
as in the proof of Theorem \ref{th:main2} we can prove the following result
which yields a better bound than the one of Proposition \ref{prop2} 
if $k \leq g$.

\begin{proposition}
	\label{prop3}
	Let $G=G(n)$ be a graph on $n$ vertices with maximum
	degree at most $\Delta=\Delta(n)$ and girth at least $g$, where $g$ is a fixed
	positive integer, and let
	$L$ be a random $(2,\mathcal{C})$-list assignment for $G$.
		If
	$\sigma(n) = \omega(n^{\frac{1}{g+1}} \Delta)$,
	then {\bf whp} $G$ is $L$-colorable.
\end{proposition}

In the next section we shall prove several other results for graphs
with girth greater than $3$.


\section{Graphs with girth greater than three}
	For the graph $H$ in the example in Section 2 showing
	that Theorem \ref{th:main1} is best possible, the threshold
	for the property of being colorable 
	from a random $(k,\mathcal{C})$-list assignment
	coincides with the threshold
	for disappearence of $(k+1)$-cliques where each vertex has
	the same list; that is, when $\sigma(n)= \omega(n^{1/k^2}\Delta^{1/k})$,
	then {\bf whp} $H$ has no such cliques, and when
	$\sigma(n) = o(n^{1/k^2}\Delta^{1/k})$, 
	then {\bf whp} $H$ contains
	a $(k+1)$-clique where the vertices have identical lists.

	As in \cite{Casselgren2, Casselgren3},
	for a graph $G=G(n)$ on $n$ vertices with girth $g \geq 4$
	(and thus with no $(k+1)$-cliques if $k \geq 2$)
	it is possible to establish
	a better bound on $\sigma(n)$ than those given
	by Theorem \ref{th:main1},
	Propositions \ref{prop1} and \ref{prop2},
	which implies list colorability of $G$
	from a random 
	$(k,\mathcal{C})$-list assignment.
	Indeed, the coarse 
	threshold $n^{1/k^2} \Delta^{1/k}$ in Theorem \ref{th:main1}
	is essentially due to the fact that the probability that a $(k+1)$-clique
	has a list assignment where all lists are equal 
	is $\binom{\sigma}{k}^{-k}$
	and the maximum number of $(k+1)$-cliques in a graph with maximum
	degree $\Delta$ is roughly $n \Delta^k$, if $\Delta$ is sufficiently
	small compared to $n$. 
	
	If $L$ is a $k$-list assignment for a graph $G$, then
	any $L$-critical subgraph of $G$ which is not a $(k+1)$-clique
	has strictly more than $k+1$ vertices, and  employing this fact
	we can use the very same methods as in Section 2 and 3,
	to prove better bounds on $\sigma$ which ensures that {\bf whp}
	$G$ is colorable from a random $(k,\mathcal{C})$-list assignment.
	Considering graphs with large girth is one way of increasing the minimum
	number of vertices in an $L$-critical graph - 
	which is the most relevant parameter -
	and one can of course derive corresponding results for
	other families of graphs.
	
	We emphasize that the results in this section are probably
	not best possible
	for any values of $\Delta$ or $k$. To prove sharp results, a first
	step would be to investigate how many 
	vertices the smallest non-$k$-choosable graph with girth $g$
	has, and then estimate the number of $k$-list assignments
	that do not contain
	a proper coloring of such a graph.

	We will proceed as in the proof of Theorem \ref{th:main1}
	and use first moment
	calculations to show that if $L$ is a random $(k,\mathcal{C})$-list
	assignment for a graph $G=G(n)$ on $n$ vertices with girth $g$, 
	where $k$ and $g$ are fixed integers satisfying
	$k \geq 3$ and $g> 3$, respectively, $\Delta=\Delta(n)$ is sufficiently small,
	and $\sigma=\sigma(n)$ is large enough,
	then {\bf whp} $G$ has no connected induced $L$-critical subgraph
	and thus is {\bf whp} $L$-colorable. (The case $k=2$ was considered
	in the previous section.)
	
	Now, any $L$-critical subgraph of $G$ has minimum degree $k+1$
	(which follows from the list-coloring version of Brooks' theorem
	since $k \geq 3$),
	and as pointed out in \cite{Casselgren3} (see also \cite{Bollobas}),
	such a graph contains
	at least $Q(k+1)$ vertices,
	where $Q(k) = Q(k,g)$ satisfies
	
\begin{equation}
\label{eq:Qk}	
	Q(k) = 
	\begin{cases}
	1+ k \Big(1+(k-1)+ (k-1)^2 + \dots + (k-1)^{\frac{g-3}{2}}\Big), & 
	\text{if $g$ is odd} , \\
	2\Big(1+(k-1)+(k-1)^2+\dots+(k-1)^{\frac{g-2}{2}}\Big), & 
	\text{if $g$ is even},
	\end{cases}
\end{equation}	
and thus
	$$
	Q(k) = 
	\begin{cases}
	1+ \dfrac{k}{k-2}\big((k-1)^{(g-1)/2}-1\big) & \text{if $g$ is odd,} \\
	\dfrac{2}{k - 2} \big( (k-1)^{g/2}-1 \big) & \text{if $g$ is even.}
	\end{cases} 
	$$
	
	Suppose that 
	$\sigma = \omega\left(n^{\frac{1}{(k-1)Q(k+1) + 1}} \Delta^s \right)$,
	where $s$ is a fixed integer satisfying
	$s \geq 1 + \frac{1}{k-1}$
	and 
	$\Delta = O\left(n^{\frac{1}{k Q^2(k+1) }} \right)$.
	By proceeding as in the proof of
	Theorem \ref{th:main1}, it is not hard to show
	that if $F$ is $L$-critical and thus by Lemma \ref{lem:charac}
	belongs to an $L$-bad proper triple $(F,v,R)$ of $G$, 
	then {\bf whp}
	$F$ contains at most $\Delta^{(k-1)Q(k+1) +1}$ vertices.
	Hence, by Lemma \ref{lem:charac}, it suffices to prove that
	{\bf whp} $G$ contains no bad proper triple $(F, v, R)$
	such that $$Q(k+1) \leq |V(F)| \leq \Delta^{(k-1)Q(k+1) +1}.$$

	Set $h(k) =  \Delta^{(k-1)Q(k+1) +1}$,
	and let $Z_m$ be a random variable
	counting the number of $L$-bad proper triples $(F, v, R)$ in $G$ 
	such that $F$ has $m$ vertices,
	and set $$Z = \sum_{m=Q(k+1)}^{h(k)} Z_m.$$
	Using Lemmas \ref{lem:numbersubgraphs} and \ref{lem:bad} 
	we may now conclude that
\begin{align*}
	\mathbb{P[\text{$G$ contains an $L$-bad proper triple}]} & \leq
	\mathbb{E}[Z] + o(1) \nonumber \\
	& \leq \sum_{m=Q(k+1)}^{h(k)} 
	n \Delta^{m-1}(m-1)!
	\frac
	{\sigma^{m-1} \binom{\Delta}{k} \binom{\Delta k}{k-1}^{m-1} }
	{\binom{\sigma}{k}^m} + o(1)
	\nonumber \\
	&	=O\left(\frac{n}{\sigma \Delta}\right) \sum_{m=Q(k+1)}^{h(k)}
	 \Big(\frac{m k^k \Delta^{k}}
	{\sigma^{k-1}} \Big)^m
	\nonumber \\
	&
	= O \Big( \frac{n \Delta^{kQ(k+1)}}{\sigma^{(k-1)Q(k+1)+1}} \Big)
	\sum_{m=0}^{\infty} \Big(\frac{k^{k} \Delta^{(k-1)Q(k+1) + k+1}}
	{\sigma^{k-1}} \Big)^m \nonumber \\ 
	& = o(1),
\end{align*}
	provided that 
	$\sigma(n) = \omega\left(n^{\frac{1}{(k-1)Q(k+1) + 1}} \Delta^s \right)$,
	$s\geq 1 + 1/(k-1)$
	and 
	$\Delta = O \left(n^{\frac{1}{k Q^2(k+1)}} \right)$.
	Let us determine $(k-1) Q(k+1)+1$ explicitly for some small values of $g$.
	When $g=4$, then if
	$$\sigma(n) = \omega(n^{\frac{1}{2k^2-1}} \Delta^s)
	\text{ and } \Delta= O\left(n^{\frac{1}{4(k^3 + 2k^2 + k)}}\right),$$
	then  {\bf whp} there
	is an $L$-coloring of $G$,
	and when $g=5$, then it suffices to require that
	$$\sigma(n) = \omega(n^{1/(k^3+k^2-1)} \Delta^s)
	\text{ and }
	\Delta = O\left(n^{\frac{1}{k^5+ 4k^4 + 8k^3+ 8k^2 + 4k}}\right).$$

	We collect these results in the following proposition.

\begin{proposition}
\label{prop:girth}	
	Let $G= G(n)$ be a graph on $n$ vertices 
	with maximum degree $\Delta$
	and girth at least $g$, where $g$ is a fixed positive integer.
	Suppose that $k$ is a fixed integer satisfying $k \geq 3$ and
	that $L$ is a random $(k,\mathcal{C})$-list assignment for $G$.
	Moreover, let  $s$ be a fixed integer satisfying
	$s \geq 1 + \frac{1}{k-1}$
	\begin{itemize}
	
	\item[(i)] If $g=4$,
	$\sigma(n) = \omega\left(n^{\frac{1}{2k^2-1}} \Delta^s \right)$, 
	and $\Delta= O\left(n^{\frac{1}{4k^3 + 8k^2 + 4k}}\right)$,
	then {\bf whp}
	$G$ has an $L$-coloring.
	
	\item[(ii)] If $g=5$,
	$\sigma(n) = \omega\left(n^{1/(k^3+k^2-1)}\Delta^s \right)$,
	and $\Delta = O\left(n^{\frac{1}{k^5+ 4k^4 + 8k^3+ 8k^2 + 4k}}\right)$,
	then {\bf whp}
	$G$ has an $L$-coloring.
	
	\item[(iii)] If $g > 5$, then there are polynomials $P(k)$ 
	and $R(k)$
	in k of
	degree $\lceil g/2 \rceil$ and $2\lceil g/2 \rceil-1$,
	respectively, such that if 
	$\sigma(n) = \omega\left(n^{1/P(k)} \Delta^s \right)$,
	and $\Delta = O(n^{1/R(k)})$,
	then {\bf whp} $G$ is $L$-colorable.
	Moreover, 
	$$P(k)=(k-1)Q(k+1)+1 \text{ and } 
	R(k)= k Q^2(k+1),$$
	where $Q(k)$ is given by \eqref{eq:Qk}.
\end{itemize}

\end{proposition}

For graphs with no restriction on the maximum degree,
we shall prove the following, which yields a
slightly weaker bound on $\sigma$. 

\begin{proposition}
\label{prop:girth2}
	Let $G= G(n)$ be a graph on $n$ vertices 
	with maximum degree $\Delta = \Delta(n)$
	and girth at least $g$, where $g > 3$ is a fixed positive integer.
	Suppose further that $k$ is a fixed integer satisfying $k \geq 3$ and
	that $L$ is a random $(k,\mathcal{C})$-list assignment for $G$.
	If $\sigma = \omega\left(n^{\frac{1}{Q(k)-1}} \Delta\right)$,
	where $Q(k)$ is given by \eqref{eq:Qk},
	then {\bf whp} $G$ is $L$-colorable.
\end{proposition}

For the proof of the above result we need some new tools.
The {\em distance} between two vertices in a graph $G$
is the number of edges in a shortest path in $G$ between them.

Given a graph $G$ with girth $g$, where $g$ is a positive integer,
an {\em odd rooted $k$-proper tree} in $G$
is a rooted tree $T$ with root $v$ such that $v$ has $k$ neighbors
in $T$, and for each $i=1, \dots, \lfloor (g-3)/2 \rfloor$, each
vertex at distance $i$ from $v$ has exactly $k-1$ neighbors at distance
$i+1$ from $v$, and no vertex of distance $\lfloor (g-1)/2 \rfloor$ 
from $v$ has a neighbor at greater distance from $v$. 
Note that the subgraph of $G$ induced by all vertices
at distance at most $\lfloor (g-3)/2 \rfloor$ from $v$ in $T$ is a tree, and
that no two vertices in $G$ at distance $\lfloor (g-3)/2 \rfloor$ from $v$ 
has a common neighbor at distance $\lfloor (g-1)/2 \rfloor$ from $v$,
because $G$ has girth $g$.

For an even integer $g$, and a graph $G$ with girth $g$
we define the concept of an {\em even $k$-proper tree} as follows:
let $u$ and $v$ be adjacent vertices and let $T_u$ and $T_v$
be two vertex-disjoint odd rooted $k$-proper trees in $G$
with roots $u$ and $v$, respectively, 
except for the fact that $u$ and $v$ has exactly $k-1$ neighbors in
$T_u$ and $T_v$, respectively.
An {\em even rooted $k$-proper tree} with root $v$ and {\em semiroot $u$} is the 
union of $T_u$ and $T_v$ along with the edge $uv$.

It is straightforward to prove that for odd $g$, an odd $k$-proper rooted 
tree has exactly $Q(k)$ vertices, 
and that for even $g$, an even $k$-proper rooted tree has exactly $Q(k)$
vertices,
where $Q(k)$ is given by \eqref{eq:Qk}. (Indeed, this formula can be proved
by considering a breadth-first search tree with root at some vertex
in a graph with minimum degree $k$.)

Suppose now that $L$ is a $k$-list assignment for $G$, and
that $T$ is an odd rooted $k$-proper tree $T$ in $G$ with root $v$.
Then $T$ is {\em $(\text{tree},L)$-bad} (or just {\em tree-bad}) if there is an $L$-coloring
$\varphi$ of $T-v$ such that

\begin{itemize}

\item[(i)] $L(v) =\{\varphi(x): x \in N_T(v)\}$

\item[(ii)] for every $i=1,\dots,\lfloor (g-3)/2 \rfloor$,
every vertex $x \in V(T)$
at distance $i$ from $v$ satisfies that $$L(x) \setminus \{\varphi(x)\} =
 \{\varphi(y) : y \in N_T(x) \text{ and 
the distance between $y$ and $v$ is $i+1$}\}.$$

\end{itemize}

Suppose now that $g$ is even and that
$T$ is an even rooted $k$-proper tree with root $v$ and semiroot $u$.
Then $T$ is 
{\em $(\text{tree},L)$-bad} (or just {\em tree-bad}) if
there is an $L$-coloring
$\varphi$ of $T-v$ such that

\begin{itemize}

\item[(iii)] $L(v) =\{\varphi(x): x \in N_T(v)\}$

\item[(iv)] for every $i=1,\dots, (g-4)/2$,
every vertex $x \in V(T)$ at distance $i$ from $v$ satisfies that
$$L(x) \setminus \{\varphi(x)\}= \{\varphi(y) : y \in N_T(x) \text{ and 
the distance between $y$ and $v$ is $i+1$}\},$$
and for every $i=1,\dots,(g-4)/2$,
every vertex $z \in V(T)$ at distance $i$ from $u$ satisfies that
$$L(z) \setminus \{\varphi(z)\}= \{\varphi(w) : w \in N_T(z) \text{ and 
the distance between $w$ and $u$ is $i+1$}\}.$$

\end{itemize}

\begin{lemma}
\label{lem:badtree}
	Let $G$ be a graph with girth $g$ and $L$ a $k$-list assignment for $G$,
	where $g$ and $k$ are  positive integers.
	If $G$ is not $L$-colorable and $g$ is odd (even), then
	there is a tree-bad odd (even) rooted $k$-proper tree $T$
	in $G$.
\end{lemma}
\begin{proof}
	We shall prove the lemma in the case when $g$ is odd. The case
	when $g$ is even can be done similarly.
	If $G$ is not $L$-colorable, then it contains a vertex-minimal
	conncected induced $L$-critical subgraph $H$. 
	Let $v$ be a vertex of $H$; then
	$H-v$ has an $L$-coloring $\theta$. For such an $L$-coloring $\psi$
	of $H-v$, let
	$W^{\psi}$ be the set of vertices $u$ in $H$ for which
	there is an $(\psi,L)$-alternating path of length
	at most $\frac{g-1}{2}$ with origin at $v$ and terminus $u$. 
	For a vertex $u$ in $H[W^{\theta}]$ we say that $u$ is {\em at level $r$}
	if the distance between $u$ and 
	$v$ is $r$.
	We shall prove by contradiction that
	$H[W^{\theta}]$ contains
	an tree-bad odd rooted $k$-proper tree $T$.
		
	If $\psi$ is an $L$-coloring of $H-v$,
	then the subgraph of $H[W^{\psi}]$ induced by all vertices
	at level at most $(g-3)/2$ in $H[W^{\psi}]$ is a tree and
	no two vertices at level $(g-3)/2$ 
	has a common neighbor at level $(g-1)/2$,
	because $G$ has girth $g$.
	
	Now consider the subgraph $H[W^{\theta}]$.
	For each color $c \in L(v)$, there is clearly a neighbor of $v$ 
	that is colored $c$ under $\theta$
	(since otherwise $H$ is $L$-colorable).
	Suppose that $H[W^{\theta}]$ does not 
	contain a tree-bad odd rooted $k$-proper tree.
	We shall prove that this implies that $H$ is $L$-colorable,
	establishing the required contradiction.
	Let $\mathcal{T}$ be the set of all subgraphs $T_{\theta}$
	of $H[W^{\theta}]$ such that
	
	\begin{itemize}
	
	\item[(a)] $v \in V(T_{\theta})$, 
	
	\item[(b)] either $0$ or exactly $k$ neighbors of $v$ 
	are in $T_{\theta}$, and
	in the latter case $L(v)=\{\theta(x): x \in N_{T_{\theta}}(v)\}$,
	
	\item[(c)] for every vertex $u$ of $T_{\theta}$ at level $q$, 
	where $1 \leq q \leq \frac{g-3}{2}$,
	$T_{\theta}$ contains exactly $0$ or
	$k-1$ neighbors of $u$ at level $q+1$,
	
	\item[(d)] for every vertex $u \in V(T_{\theta})$
	and every color $c \in L(u) \setminus \{\theta(u)\}$, if
	$u$ is at level $q$, where $1 \leq q \leq \frac{g-3}{2}$, 
	and $u$ has
	$k-1$ neighbors in $T_{\theta}$ at level $q+1$,
	then there is a neighbor $x$ of $u$ at level
	$q+1$ with $\theta(x) = c$.
	
	\end{itemize}
	We choose an element $T_{\theta}$ from 
	$\mathcal{T}$ such that the shortest
	maximal path with origin at $v$ in $T_{\theta}$ has maximum
	length. If $\mathcal{T}$ contains several elements with
	shortest maximal paths of equal length, then we choose $T_{\theta}$
	to be an element of $\mathcal{T}$ with the minimum number
	of such paths.
	We say that such a subgraph of $H[W^{\theta}]$
	is {\em path-maximal}.
	
	If each maximal path in $T_{\theta}$ with origin at $v$ has length
	at least $\frac{g-1}{2}$, then $H[W^{\theta}]$ 
	clearly contains a tree-bad odd rooted
	$k$-proper tree; so suppose
	that there is some maximal path in $T_{\theta}$ with origin at $v$
	of length strictly less than $\frac{g-1}{2}$.
	%
	Let $P=v w_1 w_2\dots w_j$ be such a shortest path in $T_{\theta}$
	and suppose further that $\theta(w_i) = c_i$, $i=1,\dots,j$.
	
	Now, if $v = w_j$,
	then $H$ is $L$-colorable; 
	a contradiction and the desired result follows.
	So suppose that $v \neq w_j$. 
	Since $T_{\theta}$ is path-maximal,
	there is a color $c_{j+1} \in L(w_j) \setminus \{\theta(w_j)\}$ 
	such that no vertex in $H[W^{\theta}]$ at level $j+1$
	is adjacent to $w_j$ and colored $c_{j+1}$. We call
	such a color a {\em free color} of $w_j$. In fact,
	since $T_{\theta} \in \mathcal{T}$ and $T_{\theta}$ is
	path-maximal, 
	there is some color $c \in L(w_{j-1}) \setminus \{\theta(w_{j-1}) \}$
	such that each neighbor of $w_{j-1}$ at level $j$ 
	in $H[W^{\theta}]$ colored
	$c$ has a free color.

	Let $H'$ be the subgraph of $H[W^{\theta}]$ consisting of
	all vertices $u$ for which there is a $(\theta,L)$-alternating path
	of length at most $j$ with origin at $v$, terminus at $u$
	and whose second vertex is colored $c_1$.
	Note that $H'$ is a tree and denote by $\{a_1, \dots, a_r\}$
	the set of neighbors of $v$ in $H'$. 
	For $i=1,\dots,r$, denote by $H'_i$ the subgraph of $H'$
	consisting of
	all vertices $u$ for which there is a $(\theta,L)$-alternating path
	of length at most $j$ with origin at $v$, terminus at $u$
	and whose second vertex is $a_i$.
	
	For $i=1,\dots,r$, let $T_i$ be a path-maximal tree in $H'_i$
	satisfying (a) and (c)-(d) (with $T_i$ in place of $T_{\theta}$)
	and the additional condition that $a_i \in V(T_i)$. 
	Since $T_{\theta} \in \mathcal{T}$ and $T_{\theta}$ is path-maximal,
	it follows that each $T_i$ contains a maximal path
	$P_i = v z^{(i)}_{1} \dots z^{(i)}_{j_i}$ of length at most 
	$j$ such that $a_i = z^{(i)}_{1}$, and 
	$z^{(i)}_{j_i}$ has a free color; and,
	more generally, there is a color 
	$c^{(i)} \in L(z^{(i)}_{j_i-1}) \setminus \{ \theta(z^{(i)}_{j_i-1}) \}$, 
	such that any vertex at level 
	$j_i$ in $H_i'$ that is
	adjacent to $z^{(i)}_{j_i-1}$ and colored $c^{(i)}$ has a free color.
	We call such a path in $T_i$ {\em good}.

	We shall now prove that
	there is an $L$-coloring $\varphi$ of $H-v$ such that 
	$v$ is not adjacent to any vertex colored $c_1$,
	and thus there is an $L$-coloring of $H$.

	For $i \in \{1,\dots, r\}$, a {\em colorful tree} $J$ with root $a_i$ is 
	a tree in $H'_i-v$ such that:
	
	\begin{itemize}
		
		\item $a_i \in V(J)$ and
		all vertices of $J$ lie on $(\theta,L)$-alternating paths
		with origin at $v$, and
	
		\item if a vertex $u \in V(J)$ at level $q$, $1 \leq q \leq j-1$
		has a neighbor in $J$ at level $q+1$ colored $c$ under $\theta$, then
		$J$ contains all neighbors of $u$ at level $q+1$ with color
		$c$ under $\theta$; and, moreover, all neighbors of $u$ in $J$ at level
		$q+1$ is colored $c$ under $\theta$.
	
	\end{itemize}
	
	A colorful tree $J$ with root $a_i$ in $H'$ is {\em correct} 
	if for each maximal
	path $P'$ in $J$ with origin at $a_i$, the terminus of $P'$
	has a free color.

	\begin{claim}
	\label{cl:corrcoltree}
		For $i=1,\dots, r$, the graph $H'_{i}$ contains a correct colorful tree.
	\end{claim}
	
	This claim can easily be proved by using the fact that 
	each $T_i$ contains a good path.
	Since suppose that
	there is no correct colorful tree in $H'_i$ for
	some $i \in \{1,\dots,r\}$.
	Then every colorful tree in $H'_i$ is not correct, which
	contradicts that $T_i$ satisfies (a), (c), (d), is path-maximal, 
	and has a good path.
	
	\bigskip
	
	By the claim above, for $i=1,\dots,r$, $H'_i$ contain a correct 
	colorful tree $J_i$.
	By recoloring all leaves of each $J_i$ with its free color, recoloring
	all other vertices of $J_i$ with the (unique) color of its children,
	and retaining the color of every other vertex of $H-v$,
	we obtain an $L$-coloring $\varphi$ of $H-v$ such that 
	$v$ is not adjacent to any vertex colored $c_1$,
	implying the required contradiction.
	\end{proof}

\begin{remark}
If $T$ is a rooted $k$-proper tree with root $v$ in a graph $G$, and $L$
is a list assignment for $G$ such that $T$ is tree-bad, 
and we define $R(x)$ to be distance
between $v$ and $x$ for any vertex $x$ of $T$, then $(T,v,R)$ is an $L$-bad
proper triple according to the definition in Section 2. So for graphs
with girth $g > 3$, Lemma \ref{lem:badtree} provides
a stronger characterization of which list assignments do not
contain a proper coloring of the graph compared to Lemma \ref{lem:charac}.
\end{remark}

Let us now prove Proposition \ref{prop:girth2}.

\begin{proof}[Proof of Proposition \ref{prop:girth2}.]
	Let $G= G(n)$ be a graph on $n$ vertices
	with maximum degree $\Delta = \Delta(n)$
	and girth at least $g$, where $g$ is a fixed positive integer.
	Suppose further that $k$ is a fixed integer satisfying $k \geq 3$ and
	that $L$ is a random $(k,\mathcal{C})$-list assignment for $G$.
	We need to prove that if  
	$\sigma = \omega\left(n^{\frac{1}{Q(k)-1}} \Delta\right)$,
	where $Q(k)$ is given by \eqref{eq:Qk},
	then {\bf whp} $G$ is $L$-colorable.
	By Lemma \ref{lem:badtree}, it suffices to prove that
	{\bf whp} $G$ does not contain a tree-bad odd (even) rooted
	$k$-proper tree if
	$g$ is odd (even).
	
	We will prove the proposition in the case when $g$ is odd;
	the case when $g$ is even is similar.
	
	As pointed out above, a rooted $k$-proper tree
	has $Q(k)$ vertices, where $Q(k)$ is given by
	\eqref{eq:Qk}. Hence, the number of odd rooted 
	$k$-proper trees in $G$ is at most $n \Delta^{Q(k)-1}$.
	Given such a tree $T$ with root $v$, the number of ways of choosing
	the list assignment $L$ for $T$ such that $T$ is $(L,\text{ tree})$-bad
	is at most 
	$$\sigma^{Q(k)-1} \binom{\sigma-1}{k-1}^{k(k-1)^{\frac{g-3}{2}}},$$
	because there are less than $\sigma^{Q(k)-1}$ ways of choosing
	a proper coloring $\varphi$ of $T-v$, and there
	are thereafter at most $\binom{\sigma-1}{k-1}^{k(k-1)^{\frac{g-3}{2}}}$
	ways of choosing the remaining colors of the lists for the vertices
	at distance $\frac{g-1}{2}$ from $v$. Note that by conditions (i)
	and (ii) above, the list of every vertex at 
	distance at most $\frac{g-3}{2}$
	from $v$ in $T$ is determined, as soon as we have chosen 
	the coloring $\varphi$
	of $T-v$.
	
	Denote by $X$ a random variable counting the number
	of tree-bad odd rooted $k$-proper trees in $G$.
	By the preceding paragraphs we have
	
	\begin{align*}
	\mathbb{E}[X]  & \leq \frac{n \Delta^{Q(k)-1}
				\sigma^{Q(k)-1} \binom{\sigma-1}{k-1}^{k(k-1)^{\frac{g-3}{2}}}}
	{\binom{\sigma}{k}^{Q(k)}}
	 \\
	& \leq A \frac{n \Delta^{Q(k)-1}  k^{2 Q(k)}}
	{ \sigma^{Q(k)-1}} \nonumber \\
	 & =o(1),
	\end{align*}
	where $A$ is some constant independent of $n$,
	because $k$ is a fixed positive integer
	and $\sigma = \omega\left(n^{\frac{1}{Q(k)-1}} \Delta\right)$.
\end{proof}

Finally, we have the following result for graphs with
	large girth.
	
	\begin{proposition}
	\label{prop:verylargegirth}
		Let $G= G(n)$ be a graph on $n$ vertices 
		with maximum degree at most $\Delta$ (where $\Delta$
		is either constant or an increasing function of $n$)
		and girth $g=g(n)$ and let
		$L$ be a random $(k,\mathcal{C})$-list assignment for $G$.
		Then
		
		\begin{itemize}
		
		\item[(i)] if $k = 2$, $g \geq C \log n$, where $C > 0$ 
		is a fixed constant,
		 and
		$\sigma(n) \geq A \Delta \log n$, where $A=A(C)$ is a fixed constant
		satisfying $A > 4e^{1/C}$,
		then {\bf whp} $G$ is $L$-colorable;
		
		\item[(ii)] for each constant $C_2 >0$, there
		are constants $k_0 = \lceil \exp(2/C_2)+1 \rceil$ and  
		$B_0 = \exp(\frac{k_0-2}{2}) k_0^2$ 
		such that if
		$g \geq C_2 \log \log n$,
		$k \geq k_0$, $B > B_0$ and
		$\sigma(n) \geq B \Delta$, 
		then {\bf whp} $G$ is $L$-colorable.
		
		\end{itemize}
	\end{proposition}
	
	\begin{proof}[Proof (sketch).]
		We first prove part (i).
		We shall prove that {\bf whp} $G$ has no $L$-coloring $\varphi$
		such that it contains a $(\varphi,L)$-alternating path of
		length at least $g$. 
		Since an $L$-critical induced subgraph of $G$ has minimum
		degree $2$, by Lemma \ref{lem:charac},
		this means that $G$ has no $L$-critical induced subgraph,		
		which implies the desired result.
		
	As in the proof of Theorem \ref{th:main1}
	the expected number of paths $P$ in $G$
	on at least
	$g$ vertices, for which there is an $L$-coloring $\varphi$ of $P-v$,
	where $v$ is the origin of $P$,
	such that $P$ is $(\varphi, L)$-alternating is at most
	
	\begin{equation}
\label{eq:pathsum2}
	 \sum_{r=g}^{n} \frac{n \Delta^{r-1}k^{2r}}{\sigma^{r-1}}
	= o(1),
	\end{equation}
	provided that $k=2$,
	$g \geq C \log n$ and $\sigma(n) \geq A \Delta \log n$,
	where  $A > 4e^{1/C}$.

		Let us now prove part (ii). Evidently, it 
		suffices to prove the theorem in the case when 
		$k = k_0$. The proof is almost identical
		to the proof Proposition \ref{prop:girth2}.
		Arguing as in that proof it suffices to prove that
		the expression
\begin{align}		
		&\frac{n \Delta^{Q(k)-1} 
				\sigma^{Q(k)-1} \binom{\sigma-1}{k-1}^{k(k-1)^{\frac{g-3}{2}}}}
	{\binom{\sigma}{k}^{Q(k)}}
	 \nonumber \\
	 \leq & \frac{n \Delta^{Q(k)-1}  k^{2 Q(k)}}
	{ \sigma^{Q(k)-1}} \label{eq:sista}
\end{align}
		tends to $0$ as $n\to\infty$.
		Noting that for large $g$, $Q(k) \geq \frac{2}{k-2}((k-1)^{g/2} -1 )$,
		it is easily verified that \eqref{eq:sista} tends to
		$0$ as $n \to \infty$ provided that
		$k \geq \exp(2/C_2)+1 $,
		$\sigma(n) \geq B \Delta$  and
		$g \geq C_2 \log \log n$,
		where $B > B_0= \exp(\frac{k_0-2}{2}) k_0^2$.
	\end{proof}


\section{Random lists of non-constant size}

	In this section we demonstrate how the methods from
	Section 2 and 3 can be used for proving some analogous results 
	on list coloring
	when the size $k$ of the lists in a random $(k,\mathcal{C})$-list
	assignment is a (slowly) increasing function of $n$.
	We shall derive such analogues of several
	of the results in the preceding
	sections.
	
	There are some previous results on coloring graphs
	from random lists of non-constant size in the literature:
	In \cite{AndrenCasselgrenOhman} it is proved that
	there 
	is a constant $c >0$ such that 
	if $L$ is a random $(k,\{1,\dots, n\})$-list,
	assignment for $\mathcal{L}(K_{n,n})$,
	where $\mathcal{L}(K_{n,n})$ is the line graph of
	the balanced complete bipartite graph on $n+n$ vertices
	and $k > (1-c)n$, then {\bf whp} there is an $L$-coloring
	of $\mathcal{L}(K_{n,n})$. Note that in \cite{AndrenCasselgrenOhman}
	this result is formulated in the language of arrays and Latin squares.
	
	In \cite{CasselgrenHaggkvist} it is proved that
	for the complete graph $K_n$ on $n$ vertices 
	the property of being colorable from a random
	from a random $(k,\{1,\dots, n\})$-list assignment 
	has a sharp threshold at
	at $k=\log n$. Moreover, a similar
	result for the line graph of the complete bipartite graph
	$K_{m,n}$ with parts of size $m$ and $n$,
	where $m =o(\sqrt n)$ is also proved.

	Let us now prove an analogue of Theorem \ref{th:main1}.
	The proofs of all of the results in this section
	follow proofs in Section 2-4 quite
	closely, so in general, we 
	omit proofs or just provide brief sketches.
	Througout this section we assume
	that $G=G(n)$ is a graph on $n$ vertices 
	with maximum degree
	at most $\Delta$, and $L$ a random 
	$(k,\mathcal{C})$-list assignment for $G$,
	where $k=k(n)$ and $\Delta = \Delta(n)$
	are non-constant increasing integer-valued functions of $n$
	satisfying $k < \Delta$.

	\begin{theorem}
	\label{th:nonconstant1}
		Suppose that $k = O(\log^{1/4} n)$ and
		$\Delta = O(n^{1/3k^3})$. For any $\epsilon >0$
		the following holds:
		
		\begin{itemize}
		
			\item[(i)] If $k = o(\log \Delta)$ and 
			$\sigma(n) \geq (1+ \epsilon)n^{1/k^2} \Delta^{1/k} k$,
		then {\bf whp} $G$ is $L$-colorable;
		
			\item[(ii)] if $k = C \log \Delta$ and 
			$\sigma(n) \geq (1+ \epsilon)n^{1/k^2} \exp(1/C) k$,
			then {\bf whp} $G$ is $L$-colorable;
			
			\item[(iii)] if $k = \omega(\log \Delta)$ and
			$\sigma(n) \geq (1+ \epsilon)n^{1/k^2} k$,
			then {\bf whp} $G$ is $L$-colorable.
		
		\end{itemize}
		
	\end{theorem}

		The proof of the above theorem is almost identical
		to the proof of Theorem \ref{th:main1}.
		Since $k$ is relatively small compared to $\sigma$ and $n$, 
		the same arguments and calculations
		as in the proof of Theorem \ref{th:main1} yield the
		required conclusions, given that $k, \Delta$ and $\sigma$
		satisfy the above conditions. The details are
		omitted.

	Note further that the example after 
	Thereom \ref{th:main1},
	shows that the bound on $\sigma$ in the above theorem is 
	best possible 
	up to the multiplicative
	factor $k$,
	provided that $k = o(\Delta^{1/2})$.

	\bigskip
	
	Next, we have the following 
	analogue of Proposition \ref{prop1}
	for random lists of non-constant size.

	\begin{proposition}
	\label{th:nonconstant2}
		Let
		$\alpha$ and $s$ be constants
		satisfying $1 < \alpha \leq 3$
		and $s \geq 2 + \gamma$, for some $\gamma >0$.
				For any $\epsilon>0$, if
		
		\begin{itemize} 
		
			\item $k = O\left(\log^{1/\alpha-\delta} n \right)$
				for some $\delta>  0$, 
			
			\item $\Delta = O\left(n^{1/k^{\alpha}}\right)$, and
		
			\item $\sigma(n) \geq (1+ \epsilon) n^{1/k^{
			\frac{\alpha+1}{2}}} \Delta^s k $,
		
		\end{itemize}
		then $G$ is $L$-colorable.

	\end{proposition}

	Again, the proof is virtually identical
	to the proof of Proposition \ref{prop1} and is
	therefore omitted.

For lists of greater size 
we have the following which is valid for all $\Delta$:

\begin{proposition} $ $ 
			\begin{itemize}
		
			\item[(i)]
			If 
		$k = o( \log n)$, then for any $\epsilon > 0$, if
		$\sigma(n) \geq (1+\epsilon) n^{1/k} \Delta k$,
			then {\bf whp} $G$ is $L$-colorable.
		
		\item[(ii)] If
		$k \geq C \log n$,
		where $C$ is some constant, then for any $\epsilon > 0$, if
		$\sigma \geq (1+ \epsilon) C \exp(1/C) \Delta \log n$, then
		{\bf whp} $G$ is $L$-colorable.
		\end{itemize}
\end{proposition}

\begin{proof}[Proof (sketch).]
	The proofs of both part (i) and (ii) are similar to the proof 
	of Proposition \ref{prop2}. 
	Arguing as in that proof, it suffices to verify that
	the expression
	$$
	n \binom{\Delta}{k}\frac{\binom{\sigma}{k} k! \binom{\sigma-1}{k-1}^k}
	{\binom{\sigma}{k}^{k+1}} = O\left(\frac{n \Delta^k k^k}{\sigma^k}   \right)
	$$
	tends to $0$ as $n \to \infty$. 
\end{proof}

\bigskip

	Finally, we remark that it is possible to
	derive corresponding results for lists of non-constant size
	of Propositions \ref{prop:girth}
	and \ref{prop:girth2} by proceeding as in the
	proofs of these propositions.
	Proposition \ref{prop:girth} is valid for non-constant $k$
	under the additional assumption that
	$$k = O \left(\log^{\frac{1}{g}-\delta} n \right), 
	\text{ for arbitrarily small $\delta >0$},$$ 
	and provided that $\sigma$ satisfies
	$$\sigma(n) \geq (1+\epsilon) n^{\frac{1}{(k-1)Q(k+1)+1}}
	\Delta^{1+ \gamma} k, \text{ for arbitrarily small
	$\gamma >0$ and $\epsilon >0$.}$$
	Similarly, Proposition \ref{prop:girth2} is valid for non-constant $k$
	provided that 
	
	\begin{itemize}
		
		\item
			$k =  o\left(  \log^{\lceil \frac{2}{g-2} \rceil}  n \right)$ and
		$\sigma(n) \geq (1+ \epsilon)n^{\frac{1}{Q(k)-1}} \Delta k^2$, or
	
		\item $k \geq C \log^{\lceil \frac{2}{g-2} \rceil} n$,
		where $C$ is some constant,
		and
		$\sigma(n) \geq 
				A
		\Delta \log^{\lceil \frac{4}{g-2} \rceil} n$,
		where $A=A(C)$ is a suitably chosen constant.
	
	\end{itemize}


\end{document}